\documentclass[12pt]{amsart}
\usepackage[all]{xy}
\usepackage{color}
\usepackage{amsthm}
\usepackage{amssymb}
\usepackage{amsxtra}
\usepackage{fullpage}
\usepackage[colorlinks=true]{hyperref}
\usepackage{tikz} 
\usetikzlibrary{cd}

\newtheorem{theorem}{Theorem}
\newtheorem*{theorem*}{Theorem}

\newtheorem{proposition}[theorem]{Proposition}

\newtheorem*{corollary*}{Corollary}
\newtheorem*{remark*}{Remark}

\theoremstyle{remark}
\newtheorem{definition}[theorem]{Definition}
\theoremstyle{remark}

\theoremstyle{remark}
\newtheorem{remark}[theorem]{Remark}
\theoremstyle{remark}

\newcommand{\co}{\colon\,}

\newcommand{\cA}{{\mathcal A}}

\newcommand{\cD}{{\mathcal D}}

\newcommand{\cF}{{\mathcal F}}
\newcommand{\cG}{{\mathcal G}}

\newcommand{\cJ}{{\mathcal J}}

\newcommand{\cL}{{\mathcal L}}

\newcommand{\cO}{{\mathcal O}}
\newcommand{\cP}{{\mathcal P}}
\newcommand{\cQ}{{\mathcal Q}}

\newcommand{\bbC}{\mathbb{C}}

\newcommand{\bbP}{\mathbb{P}}
\newcommand{\bbK}{\mathbb{K}}
\newcommand{\bbR}{\mathbb{R}}

\newcommand{\bbZ}{\mathbb{Z}}
\newcommand{\Gm}{\mathbb{G}_m}
\newcommand{\bfL}{\mathbf{L}}
\newcommand{\bfR}{\mathbf{R}}
\newcommand{\bF}{\bar{F}}
\newcommand{\bJ}{\bar{J}}

\newcommand{\Br}{\operatorname{Br}} 
\newcommand{\et}{_{\text{\textup{\'et}}}} 
\newcommand{\Gal}{\operatorname{Gal}} 
\newcommand{\GF}{\Gamma_F} 
\newcommand{\Pic}{\operatorname{\textup{\textbf{Pic}}}} 

\newcommand{\Spec}{\operatorname{Spec}} 

\newcommand{\Hom}{\mathrm{Hom}} 

\newcommand{\lp}{\textup{(}}
\newcommand{\rp}{\textup{)}}
\newcommand{\dert}{\stackrel{\bfL}{\otimes}}

\newcommand{\internalcomment}[1]{}

\DeclareMathOperator{\Ext}{Ext}

\newcommand{\proj}{\operatorname{proj}} 

\title[Derived categories]{Derived categories of curves of genus one\\
and torsors over abelian varieties}
\author{Niranjan Ramachandran and Jonathan Rosenberg}

\address{Niranjan Ramachandran, Department of Mathematics, University of Maryland, College Park, MD 20742 USA.}
\email{atma@math.umd.edu}
\urladdr{http://www2.math.umd.edu/$\sim$atma/}

\address{Jonathan Rosenberg, Department of Mathematics, University of Maryland, College Park, MD 20742 USA.}
\email{jmr@math.umd.edu}
\urladdr{http://www2.math.umd.edu/$\sim$jmr/}
\thanks{JR was partially supported by NSF grant DMS-1607162 at the beginning
  of this project.}

\subjclass[2010]{Primary 14F08; Secondary 14H52 18G80 14F22 81T30 81T35}

\keywords{torsor, Brauer group, curve of genus $1$,
twisted sheaves, derived equivalence, abelian varieties, Fourier-Mukai}

\begin{document}
\begin{abstract}
  Suppose $C$ is a smooth projective curve of genus $1$ over a perfect
  field $F$, and $E$ is its Jacobian. In the case that $C$ has no
  $F$-rational points, so that 
  $C$ and $E$ are not isomorphic, $C$ is an $E$-torsor with
  a class $\delta(C)\in H^1(\Gal(\bar F/F), E(\bar F))$. Then
  $\delta(C)$ determines a class $\beta \in \Br(E)/\Br(F)$
  and there is a Fourier-Mukai equivalence of derived categories
  of (twisted) coherent sheaves
  $\cD(C) \xrightarrow{\cong} \cD(E, \beta^{-1})$.  We generalize this to higher dimensions, namely, we prove it also for torsors over abelian varieties.  
 \end{abstract}
\maketitle

\vskip-\baselineskip
\vskip-\baselineskip
\vskip-\baselineskip

\section*{Introduction}  
There has been great interest in generalizing Mukai duality  from
abelian varieties \cite{MR607081}
to other contexts such as K3 surfaces and Calabi-Yau
varieties. (See for example \cite{doo2,MR2310257}.)
C{\u a}ld{\u a}raru \cite{ac-thesis, ac1} has proved a version of
Mukai duality for certain elliptic fibrations $X \to S$ without a
section; this states that the derived category of coherent sheaves on
$X$ is equivalent to the derived category of $\alpha$-twisted coherent
sheaves on the relative Jacobian $J\to S$ of $X$ (or
on some suitable desingularization $\bJ$ thereof). Here $\alpha$ is a
class in the Brauer group $\Br(J)$ determined by the map $X\to S$, and the
equivalence is a type of Fourier-Mukai transform.  

C{\u a}ld{\u a}raru's result is that $X$ and $J$ are twisted derived
equivalent. Recall that two varieties $V$ and $W$ over $F$ are \emph{derived
equivalent} if there is a $F$-linear triangulated equivalence between
$\cD(V)$ and $\cD(W)$, the bounded derived categories of coherent sheaves.
They are \emph{twisted derived equivalent} if there is a
$F$-linear triangulated equivalence between $\cD(V, \alpha)$ and
$\cD(W, \beta)$ of twisted derived categories, for suitable twists
$\alpha$ and $\beta$.

While C{\u a}ld{\u a}raru, following \cite{MR1242006}, was interested in
elliptic fibrations
defined over $\bbC$, we consider elliptic fibrations $C \to \Spec F$
without a section over Spec of a perfect field $F$; these correspond to
curves $C$ of genus one over $F$ without a rational point. Since $C$ is a
torsor over its Jacobian (an elliptic curve), one is led to consider the
general situation of a torsor $X$ over an abelian variety $A$ over any
perfect field $F$. The main question that we answer in this paper can be
phrased as follows:
{\it Is there a twisted derived equivalence between $X$ and $A$?
  If yes, can we describe the twisting data explicitly?} 

Since the case of torsors over elliptic curves is more accessible,
we start with it. This situation is
in fact closely analogous to the one studied by C{\u a}ld{\u a}raru,
since a genus-one curve $C$ over $F$ without rational points is a torsor
over its
Jacobian (an elliptic curve $E$), and so $C \to \Spec F$ can be viewed
as a elliptic fibration without a section.  
In this case, we provide an answer using an explicit description of the
Poincar\'e bundle on $E \times E$. This method does not work for general
abelian varieties, as such an explicit description is not available in
higher dimensions.

\subsection*{Brief description of our results} Let $A$ be an abelian
variety over a perfect field $F$ and let $X$ be a torsor over $A$. The
aim of this article is to provide a description (Theorems \ref{main},
\ref{main2}) of the bounded derived 
category $\cD(X)$ of coherent sheaves on $X$ in terms of twisted
coherent sheaves on $A$; in particular, we identify the twisting (an
element in a specific subquotient of the Brauer group of $A$) in terms
of the class of the torsor $X$. In the case of  a smooth projective
curve of genus 
one (torsor over an elliptic curve),  it will turn out that we solve
at the same time an equivalent 
problem: the description of the bounded derived category of
$\alpha$-twisted coherent sheaves on an elliptic curve, for
$\alpha$ a class in the relative Brauer group (of the elliptic
curve relative to $S=\Spec F$).  The proof is based on the Picard
stack of $X$ (viewed as a $\Gm$-gerbe on the dual abelian variety
$A^t$). One major ingredient is the Fourier-Mukai transform between
$A$ and $A^t$, slightly generalized to handle torsors; as suggested by
the referee, one needs an extension of certain results of Bondal-Orlov
to positive characteristic (see \S \ref{sec:equiv}). The other
technical input of the proof is the identification of the class of the
above-mentioned $\Gm$-gerbe in terms of the class of the torsor
$X$. This uses a specific map in a Leray spectral sequence and a clear
understanding of the close link between universal line bundles and
sections of a $\Gm$-gerbe. One way to understand our result is that
the Galois action on the closed points of $X$, under our derived
equivalence, is transformed to a gerbe on $A^t$.

\begin{remark*}
Our main result is that any torsor $X$ over an abelian variety $A$ is
twisted derived equivalent to $A^t$. 
By Balmer \cite{Balmer}, one can recover $X$ from the tensor triangulated
category $\cD(X)$. As derived equivalence disregards the tensor structure
and the category $\cD(A^t, \beta)$ does not have a tensor structure
for non-trivial $\beta$, our result does not recover the tensor structure
on $\cD(X)$. On the other hand, there is a functor
$\cD(A^t) \times  \cD(A^t, \beta) \to \cD(A^t, \beta)$
which, by our main result, becomes an action
$\cD(A^t) \times \cD(X) \to \cD(X)$ of the tensor triangulated category
$\cD(A^t)$ on the triangulated category $\cD(X)$. We do not know what is
the significance, if any, of this extra structure on $\cD(X)$. \qed
\end{remark*}

Our results are somewhat similar to, but distinct from,
other results on related problems in
\cite{MR2309993,MR3581296,MR4251609,MR2399730}.

After this paper was written, we became aware of the preprint
\cite{SankarTaylor}, which has some overlap with our results.
However, the point of view in Theorem \ref{main} here is somewhat
different, as we start with a curve $C$ without rational points
and construct a Fourier-Mukai partner $(E, \beta^{-1})$ for it, where
$E$ is the Jacobian of $C$ and $\beta$ is a gerbe determined by the
class of $C$ as a torsor over $E$.  By way of contrast, in
\cite{SankarTaylor} the focus is on trying to construct
Fourier-Mukai partners for $\Gm$-gerbes over curves of genus $1$,
and there is no discussion of generalizations to torsors over
abelian varieties of higher dimension.

At the end of this paper, we begin to explore parallels between our
main theorems and the structure of $\cD(C)$ for $C$ a projective curve
of genus $\ne 1$.  Here things are very different, since a variety
with ample canonical bundle or ample anticanonical bundle cannot have
any nontrivial Fourier-Mukai partners, by a theorem of Bondal and Orlov
\cite{MR1818984}. Nevertheless, in genus $0$, we show
in Theorem \ref{thm:genuszero} that $\cD(C)$
has a semi-orthogonal decomposition into $\cD(F)$ and
$\cD(F,\alpha)$, for some class $\alpha\in \Br(F)$.
We also briefly discuss the situation in genus $>1$, which has a very
different flavor.

\subsection{Notations} We work with (smooth) separated
schemes of finite type over a
perfect field $F$. For any scheme $Y$, we write
$\mathcal D(Y)$ for the bounded derived category of coherent sheaves
on $Y$. For any class $\alpha$ in the (cohomological) Brauer group
$\Br(Y)$, we write $\cD(Y, \alpha)$ for the bounded derived category of
$\alpha$-twisted coherent sheaves on $Y$; if $A$ is an Azumaya algebra
over $Y$ with class $\alpha \in \Br(Y)$, then as explained in
\cite[Theorem 1.3.7]{ac-thesis} or
\cite[4.3]{ac1}, one can view $\cD(Y,\alpha)$ (up to natural equivalence)
as the bounded derived category of coherent sheaves on $A$.
Alternatively, one can fix a \v Cech cocycle representing
$\alpha$ in \'etale cohomology and consider presheaves that
satisfy gluing up to this cocycle --- see \cite[\S 1]{ac-thesis}
and \cite[Lemma 2.2]{MR2571702} for an exposition.
If $\mathcal F_i$ is an $\alpha_i$-twisted coherent sheaf on $Y$ for
$\alpha_i \in \Br(Y)$, then $\mathcal F_1 \otimes \mathcal F_2$
is a $\alpha_1\cdot\alpha_2$-twisted coherent sheaf on $Y$. 

We fix an algebraic closure $\bar{F}$ of $F$ and write $\GF$ for
the Galois group $\Gal(\bar{F}/F)$. We often abuse terminology
(as many authors do) by
identifying \'etale cohomology of $\Spec F$ with the Galois cohomology,
so that $H^i\et(\Spec F, \Gm) = H^i(\GF, \bF^\times)$.
For instance, $\Br(F) = H^2(\GF, \bF^\times)$.  

\subsection{Set-up}
\label{sec:setup}
Let $C$ be a smooth geometrically connected projective curve of genus
$1$ over $F$. We are mostly interested in the case where $C$
\emph{does not} have an $F$-rational point.
(Recall that the name \emph{elliptic curve} is reserved for the
case where $C$ has an $F$-rational point and we have fixed
such a point as an origin.) 

The Jacobian $E$ of $C$ is an elliptic curve. As $C$ is a torsor for
$E$ (since they agree over $\bF$), we write $\delta(C)$ for its
class in $H^1(\GF, E(\bF))$.

The identity section
$e\co  S\to E$, $S=\Spec F$, of $E$ gives a splitting:
\[\Br(E) = \Br^0(E)\oplus \Br(F).\]
We call $\Br^0(E)$ the \emph{relative Brauer group} of $E$ (or more
exactly of $E$ rel $\Spec F$).
We will use the folklore result
(an indication of the proof of which will be given in \S \ref{sec:fund}):
\begin{theorem}[see {\cite[p.\ 122]{sl}}, \cite{as},
    {\cite[\S3.5]{AlRam}}, and {\cite[Remark 4.5.2]{MR4304038}}] 
\label{thm:Skor}
There is a natural  isomorphism 
\[\Theta\co 
\Br^0(E)  = \frac{\Br(E)}{\Br(F)} \xrightarrow{\cong} H^1(\GF,E(\bar{F})).\]
\end{theorem}

\subsection{Main Results}Our aim is a description of $\cD(C)$ in terms of $E$ and $\delta(C)$;
we also study the analogous question for torsors over abelian varieties.
\begin{theorem}[\textbf{First main Theorem}]
\label{main} 
Let $C$ be a smooth geometrically connected projective curve of genus $1$
with Jacobian $E$, over
a perfect field $F$, and let $\delta(C)$ be the class
of $C$ in $H^1(\GF, E(\bF))$. Let $\beta \in \Br^0(E)$ be the unique element with $\Theta(\beta) =\delta(C)$. 

 There is a Fourier-Mukai equivalence 
\[ \Phi\co \cD(C) \xrightarrow{\cong} \cD(E, \beta^{-1}).\]
\end{theorem} 
\begin{remark} Theorem \ref{main} for the case $F =\bbR$
  was previously proved by the second
  author \cite[Theorem 5]{jmr} and by D.\ Kussin
  \cite[Theorem 12.4]{MR3597149}, working independently,
  using different methods.  That case has a connection to physics
  discussed in \cite{jmr};  $\cD(E, \beta)$ and $\cD(C)$ 
  come from categories of D-branes on two different,
  but dual, orientifold string theories, $\beta=\beta^{-1}$ encodes the
  $O$-plane charges in one of the theories, and the equivalence is a
  mathematical formulation of the physics duality.
  The connection with physics (though not in the case of
  nonalgebraically closed fields) is also discussed in \cite{MR2399730}
  and in \cite{MR1637405}.
\end{remark}

Theorem \ref{main} admits a generalization to abelian varieties. 
Let $A$ be an abelian variety defined over a perfect field 
$F$, let $A^t$ be the dual abelian variety,
and let $Y$ be a torsor for $A$ defined over $F$ with class
$\delta(Y)\in H^1\et(S, A)$, where $S=\Spec F$. Let $\Br_1(A^t)$
denote the kernel of the map
\[\Br(A^t) \to \Br(\overline{A^t}).\]
The map $\Br(F) \to \Br_1(A^t)$ is split by the identity section of
$A^t$; one obtains 
\[ \Br_1(A^t) = \Br_1^0(A^t) \oplus \Br(F), \quad \Br_1^0(A^t)
\xrightarrow{\sim} \frac{\Br_1(A^t)}{\Br(F)}.\]
The generalization of Theorem \ref{thm:Skor} to abelian varieties
is given by the following 
\begin{proposition} 
\label{prop:BrtoH1} There is a
natural isomorphism
\[ \Theta\co \Br_1^0(A^t) \xrightarrow{\sim} H^1\et(S, A).\]
\end{proposition}
The above result uses the identity
$H^1\et(S, \Pic_{A^{t}/S}) = H^1\et(S, A)$,
which is a consequence of the fact that the N\'eron-Severi group of
$A$ is torsion-free.

Here then is the second main result of this paper:
\begin{theorem}[\textbf{Second main theorem}]\label{main2} Let
$\beta \in \Br_1^0(A^t)$ be the unique element with
$\Theta(\beta) = \delta(Y)$. 
One has a Fourier-Mukai equivalence of derived categories
\[\Phi\co \cD(Y) \xrightarrow{\cong} \cD(A^t, \beta^{-1}).\]
\end{theorem} 

\subsection*{Acknowledgements} We are grateful to discussions with J. Adams, B. Antieau, E. Aldrovandi, P. Brosnan, A. Gholampour, R. Joshua, S. Lichtenbaum, E. Mackall, M. Rapoport, S. Shankar and L. Taylor.  We would like thank a referee for pointing out that in the original version of this paper, we forgot to explain how to make things
work in characteristic $p$.

\section{Fundamentals}
\label{sec:fund}

Since we'll need it later, we review the proof of Theorem
\ref{thm:Skor}.  See also \cite[\S 2]{sl}. 
This theorem was presumably folklore for a long time;
for example, it appears without proof in \cite[equation (1)]{MR1869390}
and in \cite[equation (1)]{MR1857024}.  It's also contained
in \cite[Theorem 2.2]{as}, together with the vanishing
theorem for the elementary obstruction in the presence of an
$F$-rational point.
\begin{proof}[Proof of Theorem \ref{thm:Skor}]
As in Section \ref{sec:setup}, we let
$C$ be a smooth connected projective curve of genus $1$ over $F$,
and let $E$ be its Jacobian. The
canonical morphism $\pi\co E\to \Spec F$ induces
a Leray spectral sequence
\begin{equation}
  E_2^{p,q}=H^p\et(\Spec F, R^q\pi_*\Gm) \Rightarrow H^{p+q}\et(E, \Gm).
  \label{eq:Leray}
\end{equation}
(See for example \cite[equation (0.1)]{as} and
\cite[Theorem 12.7]{mlec}.) As usual, we think of the groups
$E_2^{p,q}$ as sitting in the first quadrant of the plane, with $p$
as the horizontal coordinate and $q$ as the vertical coordinate.
We will be interested in the terms of low degree.  
Along the bottom row, we have the
Galois cohomology of $F$, $H^p(\Spec F, \Gm)$.  For $p=1$,
this vanishes by ``Hilbert 90'' \cite[Corollary 11.6]{mlec},
and for $p=2$, this is the Brauer
group $\Br (F)$.  For $p=3$, the group $H^3(\Spec F, \Gm)$ is
somewhat mysterious.  However, in any event, since $E$ has
an $F$-rational point, given by a morphism
$e\co \Spec F\to E$ (corresponding to the identity element
$e$ for the group structure on $E$), the edge homomorphisms
$E_\infty^{p,0}\to H^p\et(E, \Gm)$ are split via $e^*$,
and so all differentials landing in the bottom ($q=0$) row
have to vanish.  Along the left-hand column, we have the groups
$H^0\et(\Spec F, R^q\pi_*\Gm) = H^q(E(\bar F),\Gm)^{\GF}$.
For $q=2$, this vanishes by Tsen's Theorem \cite[Theorem 1.2.12]{MR4304038},
\cite[Theorem 13.7]{mlec}.
So as a consequence of the spectral sequence
and vanishing of $E_2^{1,1} \xrightarrow{d_2^{1,1}} E_2^{3,0}$, we have 
an edge homomorphism
\[ \Theta\co \Br^0(E)  = \frac{\Br(E)}{\Br(F)}
\xrightarrow{\cong}  H^1(\GF, \Pic_{E/\Spec F}(\bar F)).\]
Since $\Pic_{E/\Spec F}(\bar F)$ splits
as $E^t(\bar F) \times \bbZ$ and $H^1({\GF}, \bbZ)=0$
since $\GF$ is profinite, we can replace
$H^1(\GF, \Pic_{E/\Spec F}(\bar F))$ here with
$H^1(\GF, E^t(\bar F)) = H^1(\GF, E(\bar F))$ using the
self-duality $E \cong E^t$ of $E$.  
This explains the map $\Theta$ of Section \ref{sec:setup}.
\end{proof}

\section{Equivalence Criteria}
\label{sec:equiv}

In this section we discuss criteria for telling when a Fourier-Mukai-type
functor gives an equivalence of twisted derived categories.  This is
a variant of the criteria found in \cite{BondalOrlovSemiOrthog},
\cite{MR1420564},
\cite{MR1651025}, and \cite{ac-thesis} (among others), but as a referee
pointed out to us, some methods work only in characteristic $0$, so
we explain here how to get around this assumption.  The additional
complication that comes up in characteristic $p$ is explained in
\cite[Remark 1.25]{MR2323539}.  We don't claim any great originality for
the results in this section, as they are just minor reworkings of things
in the literature, but we have found it convenient to write them down
here so that we can refer to them in the proofs of Theorems
\ref{main} and \ref{main2}.

We will use a slight modification of \cite{MR1651025}, the modification
made to allow for twisted sheaves.  Recall from
\cite[Definition 2.1]{MR1651025} that a \emph{spanning class} for a
triangulated category $\cA$ is a collection $\Omega$ of objects in $\cA$
such that for any object $a$ of $\cA$, $a=0$ if and only if
$\Hom^i_\cA(a,\omega)=0$ for all $i$ and all
$\omega\in\Omega$, and if and only if $\Hom^i_\cA(\omega,a)=0$
for all $i$ and all $\omega\in\Omega$.  As in
\cite{BondalOrlovSemiOrthog} and in \cite{MR1651025}, $\cO_x$ denotes
the skyscraper sheaf at a closed point $x$ of a scheme $X$.  The one
thing that is new in the twisted case is that when we are
dealing with $\alpha$-twisted sheaves, the endomorphism ring of the
stalk of this $\alpha$-twisted skyscraper sheaf at $x$
can be a non-trivial central simple algebra over the residue field at
$x$.  See \cite[Remark 3.3]{MR3114930} for further discussion.

\begin{proposition}[{\cite[Example 2.2]{MR1651025}}]
\label{prop:spanning}
Let $X$ be a smooth projective variety {\lp}a smooth separated projective
scheme of finite type{\rp} over a perfect field $F$, and let
$\alpha\in \Br(X)$.  Then $\{\cO_x:x\in X\}$
{\lp}here we range over the closed points of $X${\rp}
is a spanning class for $\cD(X,\alpha)$.
\end{proposition}
\begin{proof}
Bridgeland's proof in \cite[Example 2.2]{MR1651025} works with just a trivial
modification.  Since we are working with the \emph{bounded} derived category,
for any $a\ne 0$ in $\cD(X,\alpha)$, there is a maximal $q_0$ for which
$H^{q_0}(a)\ne 0$.  There must be some 
closed point $x$ in the support of $H^{q_0}(a)$ for which
$\Hom_X(H^{q_0}(a),\cO_x)\ne 0$, and then
$\Hom_\cA^{-q_0}(a,\cO_x)\ne 0$, because of the spectral sequence
$\Ext^p_X(H^{-q}(a),\cO_x)\Rightarrow \Hom_\cA^{p+q}(a,\cO_x)$.  Serre
duality (with an $\alpha$-twist) gives the result going the other way.
\end{proof}

Now we want to study when a Fourier-Mukai-type functor between categories
of the form $\cD(X,\alpha)$ is fully faithful, and when it is an equivalence.
In this regard, we can use the rest of sections 2, 3, and 4 of
\cite{MR1651025}, or alternatively \cite{MR1420564},
pretty much as is.  However, as pointed out in \cite[Remark 1.25]{MR2323539},
\cite[Theorem 1.1]{BondalOrlovSemiOrthog} and \cite[Theorem 5.1]{MR1651025}
fail in characteristic $p$, even over algebraically closed fields and
with no Brauer twist.  However, in our situation, the Fourier-Mukai functors
we study become ordinary Mukai duality \cite[Theorem 2.2]{MR607081}
after base change from $F$ to its algebraic closure, and thus become
equivalences over the algebraic closure.  So we really only need to
study the descent problem from $\bar F$ to $F$.

In situations where one just wants a fully faithful embedding, one
can also use the following adaptation of \cite[Theorem 2.3]{MR1651025}.
\begin{theorem}
\label{thm:equivcrit}
Let $X$ and $Y$ be smooth projective varieties {\lp}smooth separated
projective schemes of finite type{\rp} over a perfect field $F$.
Let $\cL$
be a $\proj_X^*\alpha^{-1}\otimes \proj_Y^*\beta$-twisted sheaf 
over $X\times_S Y$, $S=\Spec F$, and let $\Phi\co \cD(X,\alpha)\to \cD(Y,\beta)$
be the corresponding Fourier-Mukai functor
$\Phi(\cF)=\bfR\proj_{Y,*}(\cL\dert \proj_X^*(\cF))$.  Then
$\Phi$ is fully faithful if and only if the following
properties are satisfied:
\begin{enumerate}
\item For all closed points $x_1\ne x_2$ of $X$,
  $\Hom^i_{\cD(Y,\beta)}\bigl(\Phi(\cO_{x_1}),\Phi(\cO_{x_2})\bigr)=0$ for all $i$.
\item For all closed points $x$ of $X$, $\Phi$ induces isomorphisms
  for all $i$
  \[
  \Hom^i_{\cD(X,\alpha)}\bigl(\cO_x,\cO_x\bigr)\to
  \Hom^i_{\cD(Y,\beta)}\bigl(\Phi(\cO_x),\Phi(\cO_x)\bigr).
  \]
\end{enumerate}
\end{theorem}
\begin{proof}
  The functor $\Phi$ has left and right adjoints by Verdier
  duality (\cite[\S3]{MR1420564} or \cite[Lemma 4.5]{MR1651025}), and
  by its basic definition (as a composite of triangulated functors)
  is triangulated.  Thus we can apply
  \cite[Theorem 2.3]{MR1651025} and Proposition \ref{prop:spanning}.
\end{proof}
\begin{theorem}
\label{thm:equivcrit1}
Let $X$ and $Y$ be smooth non-empty geometrically connected
projective varieties {\lp}smooth separated
projective schemes of finite type{\rp} over $S=\Spec F$,
 $F$ a perfect field. Let $\cL$
be a $\proj_X^*\alpha^{-1}\otimes \proj_Y^*\beta$-twisted sheaf 
over $X\times_S Y$ and let $\Phi\co \cD(X,\alpha)\to \cD(Y,\beta)$
be the corresponding Fourier-Mukai functor
$\Phi(\cF)=\bfR\proj_{Y,*}(\cL\dert \proj_X^*(\cF))$.  Assume that
$\Phi$ and its left and right adjoints all satisfy the condition of
\textup{Theorem \ref{thm:equivcrit}}. Then $\Phi$ is an
equivalence of categories.
\end{theorem}
\begin{proof}
  This immediately follows from \cite[Theorem 3.3]{MR1651025}
  and Theorem \ref{thm:equivcrit}, along with the obvious ``twisted''
  modification of \cite[Example 3.2]{MR1651025} (using the
  assumption that $X$ and $Y$ are geometrically connected)
  to obtain that $\cD(Y,\beta)$ is indecomposable in the sense of
  \cite[Definition 3.1]{MR1651025}.
\end{proof}
The following was proved by Orlov \cite[Lemma 2.12]{doo1}
in the untwisted case.
\begin{theorem}[Descent Theorem]
\label{thm:descent}
Let $X$ and $Y$ be smooth non-empty geometrically connected
projective varieties {\lp}smooth separated
projective schemes of finite type{\rp} over a perfect field $F$.
Let $\cL$
be a $\proj_X^*\alpha^{-1}\otimes \proj_Y^*\beta$-twisted sheaf
{\lp}or complex of sheaves{\rp}
over $X\times_S Y$, $S=\Spec F$, and let $\Phi\co \cD(X,\alpha)\to \cD(Y,\beta)$
be the corresponding Fourier-Mukai functor
$\Phi(\cF)=\bfR\proj_{Y,*}(\cL\dert \proj_X^*(\cF))$.
Similarly assume
$\cL^\vee\in \cD(X\times Y, \proj_X^*\alpha\otimes \proj_Y^*\beta^{-1})$
and let $\Phi^\vee(\cG)=\bfR\proj_{X,*}(\cL^\vee\dert \proj_Y^*(\cG))$.
Then $\Phi$ and $\Phi^\vee$ are inverse equivalences if and only if
base change to the algebraic closure $\bar F$ gives inverse equivalences
$\bar\Phi\co \cD(\bar X,\bar\alpha)\to \cD(\bar Y,\bar\beta)$
and $\bar\Phi^\vee\co \cD(\bar Y,\bar\beta)\to \cD(\bar X,\bar\alpha)$.
\end{theorem}
\begin{proof}
The proof of Orlov \cite[Lemma 2.12]{doo1} applies with just trivial
modification. We review the argument for the harder direction
(that $\bar \Phi$ an equivalence implies that $\Phi$ is an
equivalence). A composition of Fourier-Mukai functors is a
Fourier-Mukai functor, and twists multiply,
so $\Phi^\vee\circ \Phi$ is represented by an
object $\cJ$ in $\cD(X\times_S X)$.  (Note that there is no twist here, since
it gives a functor from $\cD(X,\alpha)$ to itself.)  And
$\bar\Phi^\vee\circ \bar\Phi$ is represented by $\bar\cJ$, which must
be $\cO_{\bar\Delta}$, since $\bar\Phi^\vee\circ \bar\Phi$ is the
identity.  Thus $\cJ\cong \cO_{\Delta}$ by ordinary Galois descent,
and by symmetry one can do the same with $\Phi\circ \Phi^\vee$,
and so $\Phi$ and $\Phi^\vee$ are inverse equivalences.
\end{proof}

\section{Proof of the First main Theorem}
\label{sec:proof}
\begin{proof}[Proof of Theorem \ref{main}]
We proceed to the proof of Theorem \ref{main}.  Let $C$, $E$, $F$, and
$\bF$ be as before, and let $\delta(C)$ and $\Theta(\beta)=\delta(C)$. 
Our proof divides naturally into two parts. We first identify $\cD(C)$ as a twisted derived category $\cD(E, \beta')$ and then show that the twisting element $\beta' \in \Br^0(E)$ is the element $\beta$ corresponding to $\delta(C)$.  Namely, our task has two steps:  
\begin{enumerate}
\item Prove a twisted Fourier-Mukai equivalence $ \cD(C) \cong \cD(E, \beta')$ for some element $\beta' \in \Br^0(E)$; see \S \ref{twistedfm}.  
\item Prove that $\beta' = \beta$ or equivalently $\Theta(\beta') =\delta(C)$. 
\end{enumerate} 
We will prove (1) in \S \ref{twistedfm} and provide two proofs of (2), one in \S \ref{firstproof} and another in \S \ref{secondproof}.

\subsection{Twisted Fourier-Mukai functors}\label{twistedfm} 
If there were a universal (Poincar\'e) sheaf
$\cL$ on $C\times_S E$, we could define our desired equivalence by
the Mukai formula $\Phi(\cF)=\bfR\proj_{2,*}(\cL\dert \proj_1^*(\cF))$,
$\proj_1$ and $\proj_2$ the projections from $C\times_S E$ to $C$ and $E$
\cite{MR607081}. Now $\cL$ exists when $\pi\co C\to \Spec F$ has a section
(\cite[Exercise 4.3]{MR2223410}), but not otherwise
(combine \cite[Exercise 4.3 and Exercise 2.4]{MR2223410}).
So as explained in \cite[\S3.1]{MR3821178}
or in \cite[\S4]{MR3114930}, we need to consider instead
\begin{definition}\label{defn-of-p} The Picard stack 
$\cP$ of $C$ is the stack on $S=\Spec F$ whose objects
over $T$ (a scheme over $S$) are invertible sheaves on
$C\times_ST$ of degree $0$ over $C$.
\end{definition} 
Note that $\cP$ is a $\Gm$-gerbe over $E=\Pic^0 C$; see \S \ref{picstack} for details.

There is a universal sheaf $\cL$ on
$C\times_S \cP$ and there is an obstruction $\beta'\in H^2\et(E,\Gm)$
(the class of the gerbe $\cP\to E$)
to descending $\cL$ to a universal sheaf on $C\times_S E$.
In characteristic $0$, one could now apply
\cite[Theorem 3.2.1]{ac-thesis} and the proofs of
\cite[Theorem 5.1]{ac1} and \cite[Theorem 1.3]{ac2}.  To deal with the
general case, we first base change from $\Spec F$ to $\Spec \bar F$,
which kills both the class of the torsor $\delta(C)$ and the Brauer class
$\beta'$.  Hence Mukai's Theorem \cite[Theorem 2.2]{MR607081} proves that the
Fourier-Mukai functor $\Phi$ gives a derived equivalence
\[
\bar \Phi\co \cD(\bar C\cong\bar E)\xrightarrow{\cong}
\cD(\bar E, (\bar\beta')^{-1}).\]
Then, by the Descent Theorem (Theorem \ref{thm:descent}),
we get a Fourier-Mukai derived equivalence
(again given by the Mukai formula)
\[ \cD(C) \cong \cD(E, (\beta')^{-1}).\]
The inverse appears here since if one takes a $(\beta')^{-1}$-twisted
sheaf\footnote{\label{note:sh}We say ``sheaf'' but, technically, we apply
  this with a class in the derived category of a complex of sheaves.}
on $E$ and pulls it back via $\proj_2$ to $C\times_S E$,
then (left derived) tensoring with
the $\proj_2^*\beta'$-twisted universal sheaf gives an untwisted
sheaf (see footnote \ref{note:sh})  on $C\times_S E$,
to which one can apply $\bfR\proj_{1,*}$
to get an object of $\cD(C)$.
So it remains just to show that $\beta'=\beta$.
  
\subsection{An explicit description of $\Theta$ in Theorem \ref{thm:Skor}}\label{alram1} \cite[\S 3.5]{AlRam} 
  
Let $G$ be a $\Gm$-gerbe on $E$. The base change $\bar{G}$ on $\bar{E}$ is
trivial. Fix an equivalence $f\co \bar{G} \cong \bar{G_0}$ where
$G_0 =\Pic_{E/\Spec F}$ is the trivial $\Gm$-gerbe on $E$. 
Then, for any $\sigma \in \Gamma_F$, the gerbe $\sigma^*\bar{G}$ is equivalent
to $\bar{G}$ as $G$ comes from $E$. Write $f_{\sigma}$ for the resulting
self-equivalence of $\bar{G_0}$:
\[f_{\sigma}\co \bar{G_0} \xrightarrow{\,f^{-1}\,} \bar{G} \cong \sigma^*\bar{G}
\xrightarrow{\sigma^*f} \sigma^* \bar{G_0} = \bar{G_0}.\]
Any self-equivalence of $\bar{G_0}$  \cite[\S 5.1 (b)]{mgerbes} is a
translation by an element of $\bar{G_0}(\bar{F})$, i.e., by a line bundle
$L$ on $\bar{E}$. Namely, the self-equivalence is of the form
$(-) \mapsto (-) + L$. If $L_{\sigma}$ is the line bundle on $\bar{E}$
corresponding to the self-equivalence $f_{\sigma}$ of $\bar{G_0}$, then the
map $\sigma \mapsto L_{\sigma}$ is a cocycle representative for the
element $\Theta(G)$ in $H^1(\Gamma_F, \Pic_{E/\Spec F}(\bar F))$. 

\begin{remark} In order to use the explicit description of \S \ref{alram1} to prove Theorem \ref{main}, we will need to understand that 
\begin{itemize} 
\item for any smooth proper scheme $Y$ over $S=\Spec F$, the Picard scheme $\Pic_{Y/S}$ exists. 
\item universal line bundles on $Y\times_S \Pic_{Y/S}$. 
\item the Picard stack $\mathcal{P}ic_{Y/S}$ of $Y$ is a $\Gm$-gerbe on the Picard scheme $\Pic_{Y/S}$ of $Y$. 
\item  there is a natural bijective correspondence between the trivializations of this gerbe and universal line bundles. 
\end{itemize} 
Since none of the standard references on algebraic stacks contain a detailed treatment of the Picard stack $\mathcal{P}ic_{Y/S}$, we will rely on the the excellent treatment of Picard stacks found in \cite[\S 2]{brochard} for the facts recalled in \S \ref{picstack} below.
\end{remark} 

\subsection{Picard stacks and Picard schemes}\label{picstack}  Fix a smooth and proper scheme $f_Y\co Y \to S$ over $S=\Spec F$. So
the Picard scheme $\Pic_{Y/S}$ of $Y$ exists (\cite[Chapter 8]{MR1045822}
or \cite{MR2223410}).  
  
Let $T$ denote an arbitrary scheme over $S$. As usual, $B\Gm$ denotes
the stack over $S$ whose $T$-points are the Picard category of line
bundles on $T$. Let $\mathcal{P}ic_{Y/S} = \Hom_S(Y, B\mathbb G_m)$
denote the stack over $S$ whose $T$-points comprise the Picard category of
line bundles on $Y \times_S T$.
There is an exact sequence \cite[\S2.3]{brochard} of stacks over $S$
\begin{equation}\label{gerbe} 
  0 \to B\Gm \xrightarrow{f_Y^*} \mathcal{P}ic_{Y/S}
  \xrightarrow{k_Y} \Pic_{Y/S} \to 0,
\end{equation} which says that Picard stack $\mathcal{P}ic_{Y/S}$ of
$Y$ is a $\Gm$-gerbe over the Picard scheme $\Pic_{Y/S}$ of $Y$. 
For any $S$-scheme $T$, the map $f_Y^*$ is the natural pullback map
\[ B\Gm(T) \to \mathcal{P}ic_{Y/S}(T) = \Hom_S(Y, B\mathbb G_m)(T)
= B\Gm( Y\times_S T)\]
which provides the familiar exact sequence \cite[Chapter 8]{MR1045822}
\[ 0 \to \textrm{Pic}(T) \to \textrm{Pic}(Y\times_S T) \to \Pic_{Y/S}(T)\]
by evaluating \eqref{gerbe} on $T$ and taking $\pi_0$. If
$\Pic_{Y/S}(T)$ is not empty,
then we have surjectivity on the right. The sequence can be extended to
\[ 0 \to \textrm{Pic}(T) \to \textrm{Pic}(Y\times_S T) \to
\Pic_{Y/S}(T) \to \Br(T) \to \Br(Y\times_S T).\]

\subsubsection{Extensions and cohomology} The category of Picard stacks over $S$ is an abelian category. Let $G_1$ and $G_2$ be Picard stacks over $S$. Recall that there is a natural homomorphism 
\[\textrm{Ext}^1_S(G_1, G_2) \to H^1(G_1, G_2);\]
given any extension 
\[ 0 \to G_2 \to G \to G_1 \to 0,\]
one views this extension as exhibiting a $G_2$-torsor over $G_1$ and the map above sends the extension to the class of this torsor.  
In our situation, we have a map 
\[ \textrm{Ext}^1_S(\Pic_{Y/S}, B\Gm) \to H^1(\Pic_{Y/S}, B\Gm) = H^2(\Pic_{Y/S}, \Gm).\]
The image of the extension \eqref{gerbe} is the class of the gerbe
$k_Y\co \mathcal{P}ic_{Y/S} \to \Pic_{Y/S}$.  

The
sequence \eqref{gerbe} splits if and only if the gerbe
$\mathcal{P}ic_{Y/S} \to \Pic_{Y/S}$ is trivial. By \cite[Corollaire
  2.3.7]{brochard}, the gerbe $k_Y\co \mathcal{P}ic_{Y/S} \to
\Pic_{Y/S}$ is trivial if and only if there exists a universal bundle
$L$ on  $Y \times_S \Pic_{Y/S}$. 

\subsubsection{Universal line bundles} A line bundle $L$ on $Y \times_S \Pic_{Y/S}$ is universal if it represents the functor $\Pic_{Y/S}$, namely, for every affine map $U \to S$ and every  element $l$ of $\Pic_{Y/S}(U)$, one has $l = k_Y(L \big |_{Y\times_S U})$.

\subsubsection{Sections of $k_Y$ and universal line bundles}\label{sections=universal} (See \cite[Proposition 2.3.4]{brochard}.)
There exists a very natural correspondence between sections $s$ of 
$k_Y\co \mathcal{P}ic_{Y/S} \to \Pic_{Y/S}$ and universal line bundles on
$Y\times_S \Pic_{Y/S}$. Let $\mathcal L_Y$ be the universal line bundle
on $Y\times_S \mathcal{P}ic_{Y/S}$. 

By definition, a section of $k_Y\co \mathcal{P}ic_{Y/S} \to \Pic_{Y/S}$ is a map
\[s\co  \Pic_{Y/S} \to \mathcal{P}ic_{Y/S}\]
such that $k_Y\circ s$ is the identity morphism of $ \Pic_{Y/S}$.
So a section $s$
is a $ \Pic_{Y/S}$-valued point of $\mathcal{P}ic_{Y/S}$, namely, an element
$L$ of $\mathcal{P}ic_{Y/S}( \Pic_{Y/S})$ which maps to the identity
morphism $\textrm{id}_{\Pic_{Y/S}}$ of $\Pic_{Y/S}$, viewed as an element
of $\Pic_{Y/S}(\Pic_{Y/S})$.  Given $s$, consider
$\textrm{id}_Y \times s\co Y \times_S  \Pic_{Y/S} \to Y\times_S \mathcal{P}ic_{Y/S}$; it is easy to check that $(\textrm{id}_Y \times s)^*\mathcal L_Y$ is a universal line bundle $L$ on $Y \times_S \Pic_{Y/S}$.  

It is equally simple to see that, conversely, a universal line bundle gives
a section of $k_Y$. Namely, fix a universal line bundle $L$ on
$Y \times_S \Pic_{Y/S}$. For any affine map $U \to S$ and any element
$x \in \Pic_{Y/S}(U)$, i.e., a morphism $x\co U \to \Pic_{Y/S}$, the
pullback of $L$ along 
\[\textrm{id}_{Y} \times x \co Y \times_S U  \to Y \times_S\Pic_{Y/S}\]
is a line bundle on $Y\times_SU$ and hence an element $\tilde{x}$ of
$\mathcal{P}ic_{Y/S}(U)$ whose image in $\Pic_{Y/S}(U)$ is $x$. As the
assignment $x \mapsto \tilde{x}$ is clearly functorial in $U$, this gives a
section of $k_Y$. 
 
If $\mathcal L$ and $\mathcal L'$ are universal line bundles arising from
sections $s$ and $s'$, then
\[k_Y\co \mathcal{P}ic_{Y/S}( \Pic_{Y/S}) \to  \Pic_{Y/S}(\Pic_{Y/S}),
\quad k_Y(\mathcal L) = k_Y(\mathcal L') = \textrm{id}_{\Pic_{Y/S}}. \]
By the exactness of
\[ 0 \to B\Gm (\Pic_{Y/S}) \xrightarrow{f_Y^*} \mathcal{P}ic_{Y/S}(\Pic_{Y/S})
\xrightarrow{k_Y} \Pic_{Y/S} (\Pic_{Y/S}),\]
$\cL'\otimes \cL^{-1}$ is $f_Y^*L$ for some line bundle $L$ on $\Pic_{Y/S}$. 

\begin{remark} If $Y=A$ is an abelian variety, then $A(S)$ being non-empty
  says that the gerbe $k_A\co \mathcal{P}ic_{A/S} \to \Pic_{A/S}$ is trivial.
  Therefore, there exist universal line bundles on $A\times_S\Pic_{A/S}$; these,
  when restricted to $A\times_S A^t$ (here $A^t$ is the dual abelian variety),
  are the Poincar\'e line bundles. Any two Poincar\'e line bundles on
  $A \times_S A^t$ differ by a line bundle on $A^t$. \qed \end{remark} 
   
\begin{remark}\label{point-splits} Fix $x\in Y(F)$. This is a section to $f_Y\co Y \to S$ and hence the map
\[ B\mathbb G_m \xrightarrow{f_{Y}^*} \mathcal{P}ic_{Y/S}\]
admits a retraction
\[x^*\co \mathcal{P}ic_{Y/S} \to B\mathbb G_m.\]
The map
\begin{equation}\label{splitting} 
  \mathcal{P}ic_{Y/S} \to B\mathbb G_m \times_S \Pic_{Y/S} \quad M
  \mapsto  (x^*M, k_Y(M))
\end{equation} 
is an isomorphism of stacks over ${S}$. 
Thus, the gerbe $k_Y\co \mathcal{P}ic_{Y/S} \to \Pic_{Y/S}$ is trivial and so
universal line bundles exist on 
$Y\times_S \Pic_{Y/S}$. By \cite[Remarque 2.3.5]{brochard}, a
universal line bundle $\mathcal L$ provides a quasi-inverse to \eqref{splitting} if and only if $x^*\mathcal L $ is trivial; given any $\cL$,
one can obtain a quasi-inverse to
\eqref{splitting} by rigidifying $\cL$ along $x$, i.e., replacing $\cL$ by
$\cL\otimes(f_Y^*~x^*\cL)^{-1}$. \qed
\end{remark} 
\subsection{Back to the proof of Theorem \ref{main}}The stack $\cP$ (see Definition \ref{defn-of-p}) is the substack of
$\mathcal{P}ic_{C/S}$
consisting of line bundles of (fiberwise) degree zero.  Restricting
\eqref{gerbe} to $E = \Pic_{C/S}^0$ gives the exact sequence of stacks over $S$
\begin{equation}\label{gerbe0} 
0 \to  B\mathbb G_m \to \cP \to E \to 0
\end{equation} which shows that $\cP$ is a $\Gm$-gerbe over $E$. The
gerbe $\mathcal{P}ic_{C/S} \to \Pic_{C/S}$ is trivial if and only if
the gerbe $\cP \to E$ is trivial. 
By \cite[Corollaire 2.3.7]{brochard}, the gerbe $\cP \to E$ is
trivial if and only if there exists a universal line bundle on $C\times_S E$.

Recall that if a line bundle
$\cL$ on $C \times_S E$ is universal, then, for all $g\in E(\bar{F})$,
the class of the line bundle 
\begin{equation}
\cL\big |_{C \times g} 
\end{equation} on $C \cong C \times g$ is $g \in E(\bar{F})$. Let $\textrm{proj}_E: C \times_S E \to E$ denote the second projection. 
If $\cL$ is universal, then so is  its twist $\mathcal L \otimes \textrm{proj}_E^*L$ by any line bundle $L$ on $E$. 
As we saw earlier, every universal line bundle on $C\times_S E$ arises in this way, i.e., it is a twist of $\mathcal L$ by a line bundle $L$ on $E$. 

\subsubsection{Torsor action} We will need the action of $E$ on $C$
\[ E \times_S C \to C, \quad (g,y) \mapsto g\oplus y,\]
and the related ``subtraction" map
\[C \times_S C \to E, \quad (y,z)\mapsto z\ominus y.\]
satisfying
\[z\ominus y = g \iff g\oplus y =z.\]
The isomorphism \cite[Example 5.5]{conrad-av}
\begin{equation}\label{autodual}  
  E \to E^t, \quad g\mapsto \cO_E(-g)\otimes  \cO_E(e)^{-1}
\end{equation}
expresses the self-duality of $E$. 

\subsubsection{Cocycle representative}
For any $z\in C(\bar{F})$, a cocycle representative for
$\delta(C)\in H^1(\Gamma_F,\linebreak[1]
E(\bar{F}))$ is given by the assignment 
\begin{equation}\label{cocycle}
  \Gamma_F \to E(\bar{F}),\quad  \sigma \mapsto
  \sigma(z)\ominus z \in E(\bar{F}).
\end{equation}

\subsection{First method}\label{firstproof} 
We will compute $\Theta(\beta')$ using the explicit description in
\cite[\S 3.5]{AlRam}. We fix a point $x\in C(\bar{F})$.

As $E(F)$ is non-empty, the stack $\mathcal{P}ic_{E/S}$,viewed as a $\Gm$-gerbe over $\Pic_{E/S}$, is trivial. So there is a universal bundle on $\Pic_{E/S} \times_S E$. The substack $\cQ$ of $\mathcal{P}ic_{E/S}$ consisting of line bundles of degree zero; so $\cQ$ is a $\Gm$-gerbe over $\Pic_{E/S}^0 =E$. The universal (Poincar\'e) bundle \eqref{poincare} on
$E \times_S E$ provides a section of the exact sequence (obtained from (\ref{gerbe})) of stacks
\begin{equation}\label{poincare0} 0 \to B\mathbb G_m \to \cQ \to E \to 0.
\end{equation}

Let $\bar{\cQ}$ and $\bar{\cP}$ denote the base changes of $\cQ,\,\cP$ to
$\bar{E}$. The choice of an object $M_0$ in $\bar{\cP}(\bar{E})$ (namely, a
section to the morphism $\bar{\cP} \to \bar{E}$) in turn provides an
equivalence \cite[\S 5, p.\ 32]{mgerbes}  
\[f\co \bar{\cP} \to \bar{\cQ}, \quad M \mapsto \Hom_{\bar{\cP}}(M_0, M).\]
Note that composition of morphisms in $\bar{\cP}$ 
 \[\Hom_{\bar{\cP}}(K, L) \times \Hom_{\bar{\cP}}(L, M) \to \Hom_{\bar{\cP}}(K, M)\]
corresponds to the tensor product of objects (monoidal structure) in
$\bar{\cQ}$.  

For any $\sigma \in \Gamma_F$, one has $\sigma^*\bar{\cP} \cong
\bar{\cP}$ and $\sigma^* \bar{\cQ}\cong \bar{\cQ}$ as both gerbes
$\cP$ and $\cQ$ are defined over $E$. With this identification, one
has  
\[ \sigma^*f\co \bar{\cP} \to \bar{\cQ}, \quad M \mapsto
\Hom_{\bar{\cP}}(\sigma^*M_0, M).\]
It is well known  \cite[\S 5.1 (b)]{mgerbes} that the self-equivalence
$f_{\sigma}$ of $\bar{\cQ}$ 
\[
f_{\sigma} \co \bar{\cQ} \xrightarrow{f^{-1}} \bar{\cP}
\cong \sigma^*\bar{\cP} \xrightarrow{\sigma^*f} \sigma^*{\bar{\cQ}}
\cong \bar{\cQ}
\]
is a translation by an element $L_{\sigma}$ of $\bar{\cQ}$.  Our
proof will proceed by an explicit calculation of $L_{\sigma}$ below. 
 
 For any $y\in C(\bar{F})$, let $\rho_y\co \bar{C} \to \bar{E}$ be the
 isomorphism given by $z \mapsto  y\ominus z$. The pullback of the
 Poincar\'e bundle \cite{conrad-av} on $E \times_S E$: 
\begin{equation}
\label{poincare}
L=  \cO_{E\times_S E}(\Delta_E - e \times E - E \times e)
\end{equation} 
by the map $\rho_y \times id_{\bar{E}}$ gives a line bundle
\begin{equation} M_y:= \mathcal O_{\bar{C} \times_{\bar S} \bar{E}}(\Gamma_{\rho_y} - y \times \bar{E} - \bar{C} \times e).\end{equation}  
Just as $L$ provides a section of \eqref{poincare0}, $M_y$ provides a
section of \eqref{gerbe0}. The well known properties of $L$ 
\begin{equation}
  \begin{split} L\big |_{g \times E} = \cO_{g\times E}((g,g) - (g,e))
    \cong \cO_E([g] -[e]), & \quad g \neq e \\
L\big |_{E \times h} = \cO_{E \times h}((h,h) - (e,h)) \cong \cO_E([h] -[e])  & \quad h \neq e\\
L\big |_{E \times e} \cong \cO_E \cong L \big |_{e \times E}& 
\end{split}
\end{equation} 
translate into properties of $M_y$ which we write explicitly for $M_x$
\begin{equation}
\begin{split} M_x\big |_{x\oplus g \times \bar{E}}& = \cO_{x\oplus g\times \bar{E}}((x\oplus g,g) - (y\oplus g,e)) \cong \cO_{\bar{E}}([g] -[e]), \quad g \neq e \\
M_x\big |_{\bar{C} \times h} & = \cO_{\bar{C} \times h}((x\oplus h,h) - (x,h)) \cong \mathcal O_{\bar{C}}([x\oplus h] -[x])  \quad h \neq e\\
& M_x\big |_{\bar{C} \times e} \cong \cO_{\bar{C}}, \quad \quad M_x \big |_{x \times \bar{E}} \cong \cO_{\bar{E}}.
 \end{split}
\end{equation} 
For any $y$, the line bundle $M_y$ on $\bar{C} \times_{\bar S} \bar{E}$ gives a section
(also denoted $M_y$) \cite{brochard} to the morphism $\cP \to E$. Note
that the line bundle $M_y$ gives a family of line bundles on $C$
parametrized by $E$ together with a trivialization at $y \in \bar{C}$.

Observe that $\sigma^*M_x$ is $M_{\sigma^{-1} x}$. We claim the following
\[M_{\sigma^{-1}x} \otimes M_x^{-1} \cong \proj_{\bar{E}}^* L_{\sigma},
\quad \proj_{\bar{E}}\co \bar{C} \times_{\bar S} \bar{E} \to \bar{E},\]
i.e.,  the line bundle $M_{\sigma^{-1}x} \otimes M_x^{-1}$   is the
pullback of a line bundle $L_{\sigma}$ on $\bar{E}$.   This can be
proved as follows: The line bundle $M_x$ represents a section of
\eqref{gerbe0}; any two sections of \eqref{gerbe0} differ by a section
of $B\mathbb G_m$. So the two objects $M_x$ and $M_{\sigma^{-1}x}$ of
$\cP(E)$ differ by an object of $B\mathbb G_m(E)$, i.e., a line bundle
on $E$. This is the required $L_{\sigma}$.  
 
As 
\[ \proj_{\bar{E}}^* L_{\sigma}\big |_{y \times \bar{E}}  =
L_{\sigma},\] we can explicitly determine $L_{\sigma}$ by restricting
$M_{\sigma^{-1}x} \otimes M_x^{-1} $ to $y \times \bar{E}$. 
Fix $y\in C(\bar{F})$ with $y \neq x, \sigma^{-1}x$.  If $y = x\oplus g$,
then $y = \sigma^{-1}x \oplus g \oplus g_0$, where
$x = \sigma^{-1} x \oplus g_0$.  
Then $M_x$ restricted to $y\times \bar{E}$ is
$\cO_{\bar{E}}([g] -[e])$ and $M_{\sigma^{-1}x}$ restricted to $y\times \bar{E}$
is  $\cO_{\bar{E}}([g+g_0] -[e])$. 
Therefore $M_{\sigma^{-1}x} \otimes M_x^{-1}$ restricted to
$y\times \bar{E}$ is $\cO_{\bar{E}}([g+g_0] -[e] +[e] - [g])
= \cO_{\bar{E}}([g+g_0] -[g])$. 
Therefore, $L_{\sigma} = \cO_{\bar{E}}([g+g_0] -[g])$, whose
class in $E(\bar{F})$ is $g_0 = x \ominus \sigma^{-1}x$.  

If $x = \sigma(z)$, then $g_0 = \sigma(z) \ominus z$. By
\eqref{cocycle}, the assignment $\sigma \mapsto L_{\sigma}$
is a representative for the cocycle $\delta(C)$.
So $\Theta(\beta')=\delta(C)$, as required.
\end{proof}

\subsection{Second method}\label{secondproof}
This proof of Theorem \ref{main} uses the method of
``integration of a gerbe along the fibres" of \cite[\S4.3]{AlRam}. 

\subsubsection{Gerbes and their pushforwards} (See\cite[\S 3]{AlRam}.)
Fix a morphism $f\co V \to W$ of schemes over $S$ and let $A$ be an
\'etale abelian sheaf on $V$. If $\cG$ is an $A$-gerbe on $V$, its
pushforward $f_*\cG$ is a stack but not a gerbe in general. The Leray
spectral sequence for $f$ gives a map
\[H^2(V, A) \to H^0(W, R^2f_*A);\]
let $E^2_1$ be its kernel. 
A gerbe $\cG$ is called ``horizontal" \cite[\S4.3]{AlRam} if its class
$[\cG]\in H^2(V, A)$ lies in the subgroup $E^2_1$. If $\cG$ is horizontal,
then $\pi_0(f_*\cG)$ is a torsor over the \'etale sheaf $R^1f_*A$ on $W$.
This provides a map
\[E^2_1 \to H^1(W, R^1f_*A).\]
This is the map $\Theta$ in the case $\pi\co E \to S$ and $A =\Gm$.
In this case, $R^1\pi_*\Gm = \Pic_{E/S}$. Since
$H^1(S, \Pic_{E/S}) = H^1(S, E^t) = H^1(S, E)$, any horizontal
$\Gm$-gerbe on $E$ gives rise to an $E$-torsor on $S$. We can now present
the second proof of Theorem \ref{main}. 

\begin{proof} 
Let $\beta' \in \Br(E)$ be the class of the $\Gm$-gerbe $\cP \to E$.  In our
situation, a gerbe is horizontal if it becomes trivial on $E \times_S S'$
with $S' = \Spec F'$ for a finite Galois extension $F'$ of $F$.  The gerbe
$\cP$ is horizontal  (by Tsen's theorem) as it becomes trivial for any
extension $F'$ with $C(F')$ not empty. 
So we have to show that $\Theta(\beta') = \delta(C)$. 
 
As explained in \cite[\S3]{AlRam} and above, $\Theta(\beta')$ is the
class of  $E$-torsor $X=\pi_0(\pi_*(\cP))$.  So we just need to show
that $\delta(X) = \delta(C)$. 

Recall that for any \'etale sheaf $\cF$ on $E$, the \'etale sheaf
$\pi_*\cF$ on $S$ is defined as
\[ S'\mapsto H^0(E\times_S S', \cF) = \cF(E\times_SS')\]
for all finite \'etale schemes $S'$ over $S$. The sheaf $\pi_*\cF$ is
determined by the $\Gamma_F$-module $\cF(\bar{E})$.  

In our case, the pushforward $\pi_*(\cP)$ of the gerbe $k_C\co \cP \to E$
is the stack on $S$ given by
\[S' \mapsto H^0(E\times_S S', \cP).\]
Here $H^0(E\times_S S', \cP)$ is the groupoid of sections of the
gerbe $k_Y\co \cP \to E$ over $E\times_S S'$. Applying $\pi_0$ (the functor
of connected components) to it gives us the set of isomorphism classes of
sections of the gerbe over $E\times_S S'$. Therefore, the
$E$-torsor $X=\pi_0(\pi_*(\cP))$ is the set of isomorphism classes of
sections over $E(\bar{F})$. 
By \S \ref{sections=universal}, this is the same as isomorphism classes of
universal line bundles on $\bar{Y} \times_{\bar{S}}\bar{E}$. The action of
$E(\bar{F})$ is the same as in 
\S \ref{sections=universal}. From here, the argument is the same as in the
first proof presented in \S \ref{firstproof}. \end{proof} 

\section{Torsors for abelian varieties}  
In this section we discuss the generalization of
Theorem \ref{main} to higher dimensions. Again we work with
smooth proper schemes over $\Spec F$, $F$ a perfect field.

Let $Y$ be a torsor for an abelian variety $A$ over $S=\Spec F$. 
Let $\delta(Y)$ be its class in $H^1\et(S, A)$. For $X = Y$ or $X=A$, 
let ${\cP}ic_{X/S} = \Hom_S(X, B\Gm)$ (as before) denote the stack
over $S$ whose $T$-points are the Picard category of line bundles on
$X \times_S T$.  The map $k_X\co {\cP}ic_{X/S} \to \Pic_{X/S}$ is a $\Gm$-gerbe.

Let $\cP_X$ denote the substack of ${\cP}ic_{X/S}$ corresponding to
line bundles algebraically equivalent to zero. This is a $\Gm$-gerbe
over the abelian variety $\Pic^0_{X/S}$. Since $\Pic_{A/S}^0 = A^t$ and
$\Pic^0_{Y/S}$ is canonically isomorphic to $\Pic^0_{A/S}$,
we have two gerbes $k_A: \cP_A \to A^t$ and $k_Y: \cP_Y \to A^t$ over $A^t$. 
 While $\cP_A$ is a trivial gerbe, the gerbe $\cP_Y$ is not. As $Y(\bar{F})$ is non-empty, the gerbe $\cP_Y$ becomes trivial after base change to $\bar{F}$; see Remark \ref{point-splits}. Namely, the
class $\beta_Y\in \Br(A^t)$ of $\cP_Y$ actually lies in
the kernel $\Br_1(A^t)$ of the map $\Br(A^t) \to \Br(\bar{A^t})$.

We proceed to the proofs of Proposition \ref{prop:BrtoH1}
and Theorem \ref{main2} as stated in the introduction.

\begin{proof}[Proof of Proposition \ref{prop:BrtoH1}]
This follows from \cite[Proposition 4.3.2]{MR4304038}.

The proof of this is just like that of Theorem \ref{thm:Skor}, with
one major exception, namely that $E_2^{0,2}$ in the Leray spectral
sequence for $\pi\co A^t\to S$ may not vanish.  But $E_\infty^{0,2}$
turns out to be identifiable with the quotient
${\Br(A^t)}/{\Br_1(A^t)}$, so as long as we are only interested in
${\Br_1(A^t)}$, this is not an issue. We
still have vanishing of all differentials landing on the bottom row,
because of the splitting defined by $e_*\co S\to A^t$, and so
\[
E_2^{1,1}=H^1(\Gamma_F, \Pic_{A^{t}/S}(\bar F))
= H^1(\Gamma_F, NS(A^t)\times A(\bar F)) = H^1\et(S, A),\]
since the N\'eron-Severi group of an abelian variety is torsion-free,
and thus $H^1(\Gamma_F,NS(A^t))=0$, and $A$ is the dual of $A^t$.  
\end{proof}


\begin{proof}[Proof of Theorem \ref{main2}]
Recall that, for any fixed $z \in Y(\bar{F})$, the assignment 
\begin{equation}\label{Y-cocycle} 
\Gamma_F \to A(\bar{F}),\quad \sigma \mapsto \sigma(z)\ominus z
\end{equation}is a cocycle representative for $\delta(Y) \in H^1(\Gamma_F, A(\bar{F}))$. 

We follow the method of proof of Theorem \ref{main}. 

The exact sequence, for each scheme $T$ finite and \'etale over $S$, 
\[ 0 \to B\Gm(A^t\times_S T) \to \cP_Y(A^t \times_S T) \to A^t(A^t\times_S T)\]
gives an exact sequence of stacks on $S$ in the \'etale topology. 
In particular, we obtain an exact sequence of $\Gamma_F$-modules 
\[ 0 \to \textrm{Pic}(A^t\times_S \bar{S}) \to \cP_Y(A^t\times_S\bar{S}) \to A^t(A^t\times_S\bar{S}) \to 0\] 
as the gerbe splits over $\bar{S}$. Taking $\pi_0$ of the stacks
gives an exact sequence of $\Gamma_F$-modules.
The boundary map becomes 
\[ \textrm{Mor}_S(A^t, A^t) \to H^1(\Gamma_F, \textrm{Pic}(\bar{A}^t)).\]
The image of $\textrm{id}_{A^t}$ is a cocycle which measures
the obstruction to the exact sequence 
\[ 0 \to B\Gm \to \cP_Y \to A^t \to 0\]
of $S$-stacks splitting over $S$. Given any lift $s$ of $\textrm{id}_{A^t}$
over $\bar{S}$, the assignment 
\[\Gamma_F \to \textrm{Pic}(\bar{A}^t),
\quad \sigma \mapsto \sigma(s).s^{-1}\]
is a representative for this cocycle which records the (failure of)
$\Gamma_F$-equivariance of $s$. 

Let $L$ be the Poincar\'e bundle \cite[Theorem 4.16]{conrad-av} on
$A\times_S A^t$;  recall its properties:  
\begin{itemize} 
\item $L$ is trivialized along $A \times e$ and $e\times A^t$ such that the
  two trivializations coincide on the fiber $L \big|_{(e,e)}$:
\[ L\big |_{A\times e} \cong \cO_A,\quad L \big |_{e \times A^t} \cong \cO_{A^t}.\] 
\item for any $h \in A^t(\bar{F})$, the line bundle $ L\big |_{A
  \times h}$ has class $h \in \Pic^0(A) = A^t$.  
\item for any $g\in A(\bar{F})$, the line bundle $L\big |_{g\times
  A^t}$ has class $g \in  \Pic^0(A^t) = A$.  
\end{itemize} These properties characterize $L$ uniquely. 
 

For any $y \in Y(\bar{F})$, let $M_y$ be the pullback of the
Poincar\'e bundle $L$ on $\bar{A} \times_{\bar S} \bar{A}^t$ via the map 
\[\rho_y: \bar{Y} \times_{\bar S} \bar{A}^t \to \bar{A} \times_{\bar S} \bar{A}^t, \quad (z, h) \mapsto (z\ominus y, h).\]
The bundle $M_y$ provides a splitting of 
\[ 0 \to B\Gm \to \cP_Y \to A^t \to 0\]
over $\bar{F}$. One has
\[ M_y \big |_{y\times \bar{A}^t} \cong \mathcal O_{\bar{A}^t}.\]
 
Fix $x\in Y(\bar{F})$. One has $M_x$ and as before, for any $\sigma\in
\Gamma_F$, the identity $\sigma^*M_x = M_{\sigma^{-1}x}$ holds.  
Therefore, $\sigma^*M_x \otimes M^{-1}_x$ is $f_Y^* L_{\sigma}$ for
some $L_{\sigma}$ on $\bar{A}^t$; the class of $L_{\sigma}$ is an
element of $\Pic^0(\bar{A}^t) = \bar{A}$.
Consider the restrictions of $M_x$ and $\sigma^*M_x$ to $x \times \bar{A}^t$:
  \[ M_x \big |_{x \times \bar{A}^t} \cong \mathcal O_{\bar{A}^t}, \quad \sigma^*M_x  \big |_{x \times \bar{A}^t} \cong L \big |_{(\sigma^{-1}x \ominus x) \times \bar{A}^t}.\]
Therefore, $L_{\sigma}$ is isomorphic to the line bundle
  \[ L \big |_{(\sigma^{-1}x \ominus x) \times \bar{A}^t}\]
on $\bar{A}^t$. This says that the class of  $L_{\sigma}$ in
$\Pic_{A^t/S}(\bar{F})$ is given by 
the element $\sigma^{-1}x \ominus x$ in  $A(\bar{F}) =
\Pic^0_{A^t/S}(\bar{F})$ using $\Pic^0_{A^t} = (A^t)^t = A$.  
 
 If $x = \sigma(z)$, then $g_0 = \sigma(z) \ominus z$. By
\eqref{Y-cocycle}, the assignment $\sigma \mapsto L_{\sigma}$
is a representative for the cocycle $\delta(C)$.
So $\Theta(\beta_Y)=\delta(Y)$, as required.
\end{proof}




\section{Curves of other genera}
\label{sec:OtherGenera}

To put our results in context, it is useful to compare them with some
results for other varieties and with the situation for curves of genus
$g\ne 1$.  Recall that by a result of Bondal and Orlov
\cite{MR1818984}, one cannot
expect Fourier-Mukai partners for projective varieties with
ample canonical or anticanonical bundle, which explains why we have
focused on the genus $1$ situation.  Nevertheless, there is an
analogue of our main theorem in the context of curves $C$ of genus zero,
in that the curve $C$ canonically determines a Brauer group class $\alpha$
(in this case of order $2$), and the ``nontrivial part'' of the
derived category $\cD(C)$ can be identified with the derived category of
$\alpha$-twisted sheaves.

\begin{theorem}
\label{thm:genuszero}
Let $C$ be a smooth geometrically connected projective curve of genus $0$
over perfect field $F$. Then $C$ is a quadric in $\bbP^2$ and is the
Severi-Brauer variety of a quaternion algebra $A$ over $F$.
If $\alpha$ is the class of $A$ in $\Br F$, then $\cD(C)$ has a
semi-orthogonal decomposition into $\cD(F)$ and $\cD(A)\cong \cD(F,\alpha)$.
\end{theorem}
\begin{proof}
That $C$ is a plane quadric is a standard fact \cite[Lemma
  53.10.3]{Stacks}. But a homogeneous quadric in three variables
determines a quaternion algebra, and so $C$ is the
Severi-Brauer variety of a quaternion algebra $A$ over $F$. Now the
result immediately follows from \cite{MR2571702}.
\end{proof}
\begin{remark}
Note that this result about derived categories is a natural
refinement of Quillen's calculation of the $K$-theory of
Severi-Brauer varieties in \cite[\S8]{MR0338129}, which in this
context gives $\bbK(C)\cong \bbK(F) \oplus \bbK(F,\alpha)$,
where $\bbK$ denotes the $K$-theory spectrum.
\end{remark}
\begin{remark}
  Our results for curves may be viewed as analogues of results for K3
  surfaces in \cite{MR2310257},
  \cite{MR3616008}, \cite{MR3616005}, and \cite{MR4184293}.
  Question 2 at the beginning of \cite{MR3616008},
  answered for example in \cite{MR3616008},  asks if existence
  of an $F$-rational point on a K3 surface is preserved under a
  twisted derived equivalence $\cD(X,\alpha)\cong \cD(Y,\beta)$.
  We have seen that for genus-$1$ curves, the answer is an emphatic ``no.'' 
\end{remark}

Now that we've discussed curves of genus $0$ and genus $1$, it is natural
to ask about an analogue of Fourier-Mukai duality for curves of genus
$\ge 2$.  A natural program for doing this is first to see what happens
over an algebraically closed field, and then to hope that for a general
curve $C$ defined over $F$, one gets a full and faithful Fourier-Mukai
functor from $\cD(C)$ to $\cD(X,\beta)$, for suitable $X$ and
$\beta\in \Br(X\text{ rel }F)$.  One's first guess might be to
take $X$ to be the Jacobian of $C$, since one has a Poincar\'e line
bundle over $C\times X$.  Then perhaps one could use Theorem \ref{main2}
when $C$ does not have a rational point.

Unfortunately, this program runs into trouble right from the start.
If $C$ is a smooth geometrically connected projective curve of genus $\ge 2$
over an algebraically closed field $F$, then one can use the
Poincar\'e line bundle over $C\times X$ to define a Fourier-Mukai
functor $\Phi_C\co \cD(C)\to \cD(X)$, $X$ the Jacobian of $C$,
but $\Phi_C$ is \emph{not} full and faithful.  To see this one can use
the criteria in either Theorem \ref{thm:equivcrit}
or \cite[Theorem 1.1]{BondalOrlovSemiOrthog};
the second conditions of
those theorems fail, since (up to a degree shift)
$\Phi_C\left(\cF_{x_0}\right)=\cO_X$ for $\cF_{x_0}$ the skyscraper sheaf at
the basepoint $x_0$ in $C$ used to define the Jacobian embedding, and
$\Hom^{\dim X}_X(\cO_X, \cO_X)=F\ne 0$, where $\dim X>\dim C$.
It is thus necessary to abandon hope of using abelian varieties and
torsors over them to get an analogue of Theorem \ref{main}.

However, there is another variety $X$ that seems to work, at least in
the algebraically closed case \cite{MR3713871,MR3880395,MR3764066},
namely the moduli space of stable rank-$2$ vector bundles over $C$
with determinant isomorphic to a fixed line bundle of odd degree.
One then needs to replace the Poincar\'e line bundle by the universal
rank-$2$ bundle over $C\times X$.  We leave it as an interesting
problem for the future to determine in the general case
whether there is a suitable
Brauer twist of $X$ for which one gets a full and faithful
functor from $\cD(C)$ to $\cD(X,\beta)$.  Over an algebraically
closed field of characteristic $0$ at least, $X$ has trivial Brauer
group --- see \cite{MR2353678}, so maybe the Brauer twist is unnecessary.

\bibliography{Derived-genus-one}
\bibliographystyle{plain}

\end{document}